\documentclass[12pt, a4paper]{amsart}
\usepackage[utf8]{inputenc}
\usepackage{latexsym}
\usepackage{amscd, amsfonts,%
  eucal, mathrsfs, amsmath, amssymb, amsthm,%}
  tikz}
\input xy
\xyoption{all}
\usetikzlibrary{cd}
\usepackage{fullpage}
\usepackage{extarrows}
\tikzset{
  labl/.style={anchor=south, rotate=90, inner sep=.5mm}
}
\tikzset{
  labr/.style={anchor=north, rotate=90, inner sep=1mm}
}

\usepackage{mathtools}
\usepackage{bm}

\usepackage{array}
\usepackage{color}

\usepackage[cal=boondoxo, scr=boondoxo]{mathalfa}

\usepackage[sort,nocompress]{cite}

\usepackage{thmtools}
\usepackage{thm-restate}

\usepackage{hyperref}

\usepackage[capitalise]{cleveref}

\DeclareMathOperator{\supp}{supp}

\newcommand{\field}[1]{\mathbb #1}

\newcommand{\R}{\field R}

\newcommand{\C}{\field C}

\newcommand{\Z}{\field Z}

\DeclareMathOperator{\spec}{Spec}

\renewcommand{\P}{\field P}

\newcommand{\Gal}{\operatorname{Gal}}
\newcommand{\ch}{\operatorname{char}}

\newcommand{\id}{\operatorname{id}}

\DeclareMathOperator{\Aut}{\operatorname{Aut}}

\newtheoremstyle{fauxsubsub}% name of the style to be used
  {3pt}% measure of space to leave above the theorem. E.g.: 3pt
  {0pt}% measure of space to leave below the theorem. E.g.: 3pt
  {}% name of font to use in the body of the theorem
  {0pt}% measure of space to indent
  {}% name of head font
  {.}%
  {.5em}% space after theorem head; " " = normal interword space
  {}

\newtheorem{lem}{Lemma}[section]

\numberwithin{equation}{lem}

\newtheorem{thm}[lem]{Theorem}
\newtheorem*{mainthm}{Main Theorem}
\newtheorem*{corollary*}{Corollary}

\newtheorem{prop}[lem]{Proposition}

\theoremstyle{definition}
\newtheorem{defn}[lem]{Definition}

\theoremstyle{definition}
\newtheorem{remark}[lem]{Remark}
\newtheorem*{remark*}{Remark}
\newtheorem*{remarks*}{Remarks}

\newtheorem{ques}[lem]{Question}

\DeclareMathOperator{\PO}{PO}

\author{Yuya Sasaki}

\title{On real forms of Fermat hypersurfaces}

\begin{document}

\begin{abstract}
  In this paper, we compute the number of real forms of Fermat
  hypersurfaces for degree $d$ except the degree $4$ surface case,
  and give explicit descriptions of them. 
\end{abstract}

\maketitle

\section{Introduction}

A real form of a complex algebraic variety $X$ is a real algebraic
variety whose complexification is isomorphic to $X$. Now, some
examples of smooth complex varieties with infinitely many nonisomorphic
real forms are known. The first example was obtained by Lesieutre in
\cite{MR3773792} for varieties of dimension $6$, and the existence of
such varieties for any dimension $n \ge 2$ was shown in \cite{MR3934593}. 
Also, for dimension $1$, it is known that there are no such
varieties. In addition, the construction of smooth complex rational affine
surface with uncountably many nonisomorphic real forms, and some other
results about finiteness problem of real forms is summarized in
\cite{2105.08044}.

Considering the real forms of specific varieties is also an interesting
question. 
Fermat hypersurface of dimension $n$ and degree $d$ is a projective
variety defined as follows: 
\[
  F^n_{d (, k)} := \left( X_0^d + \cdots + X_n^d + X_{n + 1}^d = 0 \right)
  \subset \P_k^{n + 1}
\]
where $k$ is the base field, and we don't write it explicitly when it is
clear. In \cite{Ogu}, the following question is asked:

\begin{ques}
  Find all real forms of $F^n_{d, \C}$ up to $\R$-isomorphisms. 
\end{ques}

When $d = 1$ or $( n, d ) = ( 1, 2 )$, $F^n_d$ is isomorphic to
a projective space, and this has exactly $1$ real form if $n$ is even,
and $2$ real forms if $n$ is odd, up to isomorphisms, which is proved
in \cite{Ben16}. Also, it is proved in \cite{Rus02} that $F^2_2$ has
exactly $4$ real forms. Note that by \cite{MR181643}, for degree $2$,
$F^n_2$ has only finitely many real forms up to isomorphisms as the first
Galois cohomology of an affine linear algebraic group is finite. 

In this paper, we give the following partial answer to the above question:

\begin{mainthm}
  Suppose $( n, d ) \neq ( 2, 4 )$ and $d \ge 3$. Then, the number of
  real forms of $F^n_{d, \C}$ is
  \[
    \begin{dcases}
      \left\lfloor \frac{n}{2} \right\rfloor + 2 & ( d \equiv 1
      \pmod{2} ) \\
      \frac{( n + 3 ) ( n + 5 )}{8} & ( d \equiv 0,\ n \equiv 1
      \pmod{2} ) \\
      \frac{( n + 4 ) ( n + 6 )}{8} + 1 & ( d \equiv 0,\ n \equiv 0
      \pmod{2} )
    \end{dcases}.
  \]
  Also, the number of real forms of $F^n_{2, \C}$ is \[
    \begin{dcases}
      \frac{n + 1}{2} + 1 & ( n \equiv 1 \pmod{2} ) \\
      \frac{n}{2} + 3 & ( n \equiv 0 \pmod{2} )
    \end{dcases}.
  \]
\end{mainthm}

We prove this theorem by computing $H^1 ( \Gal ( \C / \R ), \Aut
F^n_d )$. In Section \ref{num_main}, we deal with the case $( n, d ) \neq
( 1, 3 ), ( 2, 4 )$, $d \ge 3$, in Section \ref{num_ell}, we deal
with the case $( n, d ) = ( 1, 3 )$, and in Section \ref{quad}, we deal
with the case $d = 2$. In every case, we have already known
the automorphism group $\Aut F^n_d$, as the first case is due to Shioda
\cite{Shi87}, the second is due to the fact that $F^1_3$ is elliptic,
and in the third case, $\Aut F^n_2 \cong \PO ( n + 2 )$. 
So we will check the condition for an automorphism of
$F^n_d$ to be a cocycle, compute a conjugation of a cocycle by an
automorphism, and get this result. 

We also give explicit descriptions of these real forms for the curve
case in Theorem \ref{des:curve}, and for the surface case in Theorem
\ref{des:surface} in the way we can generalize it to higher dimensions,
and give a topological description of a real locus of some real forms (in
Theorem \ref{rl:gen}, and in particular, Theorem \ref{rl:curve} for the
curve case). We also show that $F_4^2$ has at least $4$ non-isomorphic
real forms in Remark \ref{f2_4}.

This paper is a refined version of my master thesis \cite{S21} (in
Japanese, unpublished) under Professor Keiji Oguiso, together with a
correction of some mistakes in the preliminary version, which was kindly
pointed out by Professor J\'{a}nos Koll\'{a}r.

\section{Preliminary about real forms}

First, we review the definition of a real form and related concepts. 

\begin{defn}
  Let $X$ be a complex algebraic variety and $X'$ a real algebraic
  variety.
  \begin{enumerate}
    \item The complexification of $X'$ is the complex algebraic variety
      \[
        {X'}_\C := X' \times_{\spec \R} \spec \C.
      \]
    \item A real form of $X$ is a real algebraic variety $X_\R$ with
      a $\C$-isomorphism $( X_\R )_\C \xrightarrow[]{\sim} X$. 
    \item A real structure on $X$ is an anti-involution $\rho \colon
      X \to X$, i.e., an involution that makes the following diagram
      commute:
      \[
        %\begin{tikzcd}
        %  X \arrow [ d ] \arrow [ r, "\rho" ] \arrow [ rd,% phantom,
        %  "\circlearrowleft" ] & X \arrow [ d ] \\
        %  \spec \C \arrow [ r ]& \spec \C
        %\end{tikzcd}
        \xymatrix{
          X \ar[d], \ar[r]^\rho \ar@{}[dr]|\circlearrowleft & X \ar[d] \\
          \spec \C \ar[r] & \spec \C
        }
      \]
      where the bottom arrow is the morphism induced by the complex
      conjugation. 
    \item Real structures $\rho$, $\rho'$ on $X$ are equivalent if
      there exists a $\phi \in \Aut_\C X$ such that $\rho' =
      \phi^{- 1} \circ \rho \circ \phi$. 
  \end{enumerate}
\end{defn}

\begin{thm}[\cite{MR181643}]
  Let $X$ be a complex quasi-projective variety with a real
  structure $\rho$. Then, there are bijections between 
  \begin{itemize}
    \item the set of $\R$-isomorphism classes of real forms of $X$,
    \item the set of equivalent classes of real structures on $X$, and
    \item $H^1 ( \Gal ( \C / \R ), \Aut_\C X )$ where the nontrivial
      element of $\Gal ( \C / \R )$ acts on $\Aut_\C X$ by the
      conjugation of $\rho$. 
  \end{itemize}
\end{thm}

It is clear by definition that there exists a bijection between the
second and the third set, and existence of a bijection between the first
and the third set is proved in \cite[Chapter III, 1.3]{Se02}.

\begin{remark}
  In the above condition,
  \begin{itemize}
    \item $\alpha \in \Aut_\C X$ is a cocycle if and only if $( \alpha
      \circ \rho )^2 = \id_X$. 
    \item cocycles $\alpha, \beta \in \Aut_\C X$ are equivalent if and
      only if there exists $\phi \in \Aut_\C X$ such that $\beta \circ
      \rho = \phi^{- 1} \circ ( \alpha \circ \rho ) \circ \phi$,
      and we denote it by $\alpha \sim \beta$. 
  \end{itemize}
\end{remark}

\section{The number of real forms of Fermat hypersurfaces}
\label{num_main}

In order to study real forms of Fermat hypersurfaces, we first
recall automorphisms of them. Clearly, there are two types of
automorphisms when the base field is algebraically closed and of
characteristic $0$. An automorphism of the first type is the one with
which we can identify $\sigma \in \mathfrak{S}_{n + 2}$ defined as follows: \[
  F^n_d \to F^n_d; [ X_0 : \cdots : X_{n + 1} ] \mapsto [
  X_{\sigma ( 0 )} : \cdots : X_{\sigma ( n + 1 )} ]
\]
with $\sigma$ viewed as a permutation of $\{ 0, \ldots, n + 1 \}$. 

An automorphism of the second type is the one denoted by $g ( i, j )$
for $i, j \in \Z$, $0 \le i \le n + 1$ and defined as follows: \[
  F^n_d \to F^n_d; [ X_0 : \cdots X_{n + 1} ] \mapsto [ X_0 :
  \cdots : \xi^j X_i : \cdots : X_{n + 1} ]
\]
with $\xi = \xi_d$ a primitive $d$-th root of unity. Also, we denote
the product $g ( 0, j_0 ) \cdot \cdots \cdot g ( n, j_n )$ by
$D \left( j_0, \ldots, j_n \right)$.

Let $G^n_d$ be the subgroup of $\Aut F^n_d$ generated by $g ( i, j )$
for all $i, j \in \Z$. Then, the following result is known: 

\begin{thm}[\cite{Shi87}]
  Let $k$ be an algebraically closed field with $\ch k = 0$, and suppose
  $( n, d ) \neq ( 1, 3 ), ( 2, 4 )$ and $d \ge 3$. Then, \[
    \Aut F^n_{d, k} = G^n_d \cdot \mathfrak{S}_{n + 2}
  \]
  where $\cdot$ denotes a semidirect product. 
\end{thm}

In particular, $\Aut F^n_{d, \C} = G^n_d \cdot \mathfrak{S}_{n + 2}$. 

Though this theorem is proved only in the curve case in that paper, the
same argument works in higher dimensions, as we can show that every
automorphism is linear. 

From now on, we assume $( n, d ) \neq ( 1, 3), ( 2, 4 )$, $d \ge 3$, and
the base field is $\C$ unless otherwise stated. 

Let $\tilde{\rho} \colon \P^{n + 1} \to \P^{n + 1}$ be the complex
conjugation,
i.e., the $\R$-morphism defined by $\C [ X ] \to \C [ X ]$ that sends
$X_i$ to $X_i$ for all $i$ and $c \in \C$ to $\overline{c}$. This is a real
structure of $\P^{n + 1}$, and it induces the real structure $\rho \colon
F^n_d \to F^n_d$ on $F^n_d$. 

\begin{remark}
  The following relations hold for $D \left( j_0, \ldots, j_{n + 1}
  \right)$, $\sigma \in \mathfrak{S}_{n + 2}$, and $\rho$.
  \begin{align*}
    & \sigma \cdot D \left( j_0, \ldots, j_{n + 1} \right) \cdot
    \sigma^{- 1} = D \left( j_{\sigma ( 0 )}, \ldots, j_{\sigma ( n + 1 )}
    \right) \\
    & \sigma \cdot \rho = \rho \cdot \sigma \\
    & D \left( j_0, \ldots, j_{n + 1} \right) \cdot \rho = \rho \cdot
    D \left( - j_0, \ldots, - j_{n + 1} \right) = \rho \cdot D \left( j_0,
    \ldots, j_{n + 1} \right)^{- 1}.
  \end{align*}
\end{remark}

The following result is the main step in proving the main theorem.

\begin{prop}
  Let $( n, d ) \neq ( 1, 3 ), ( 2, 4)$, and $d \ge 3$. 
  \begin{enumerate}
    \item $\alpha = \mu \sigma \in \Aut F^n_d$, with $\mu = D \left( p_0,
      \ldots, p_{n + 1} \right) \in G_d^n$, $\sigma \in
      \mathfrak{S}_{n + 2}$,
      is a cocycle if and only if the following conditions are satisfied:
      \begin{itemize}
        \item $\sigma$ is of type $1^l \cdot 2^m$ with $l + 2 m = n + 2$.
        \item $p_i \equiv p_{\sigma ( i )} \pmod{d}$ for any $0 \le i
          \le n + 1$ if $\sigma$ has a
          fixed point, i.e., $\sigma ( k ) = k$ for some $k$. 
        \item $p_i - p_{\sigma ( i )} \equiv p_j - p_{\sigma ( j )}
          \equiv 0, \dfrac{d}{2} \pmod{d}$ for any $0 \le i, j \le
          n + 1$ if $\sigma$ has no fixed point, i.e., $\sigma$ is of
          type $2^{\frac{n + 2}{2}}$. 
      \end{itemize}
      \label{fh:cocycle}
    \item Cocycles $\beta = \lambda \tau$ with $\lambda \in G^n_d$ and
      $\tau \in \mathfrak{S}_{n + 2}$ and $\gamma$ are equivalent if and
      only if there exist $\psi \in \mathfrak{S}_{n + 2}$ and $r_j \in \Z$
      for $0 \le j \le n + 1$ with $r_j = r_{\sigma ( j )}$ for all
      $j$ and $r_l \equiv 0 \pmod{2}$ for all $l \notin \supp \tau$
      such that
      \[
        \gamma = \psi^{- 1} \cdot \left( D \left( r_0, \ldots, r_{n + 1}
        \right) \cdot \beta \right) \cdot \psi.
      \]
      Here, $\supp \tau$ is the support of a permutation $\tau$,
      i.e., $\supp \tau := \left\{ i | \tau ( i ) \neq i \right\}$.
      \label{fh:cocycle_equiv}
  \end{enumerate}
  \label{fh}
\end{prop}

\begin{proof}
  (\ref{fh:cocycle}) By definition, $\alpha$ is a cocycle if and only if
  $( \mu \sigma \rho )^2 = \id$. Also, 
  \begin{align*}
    ( \mu \sigma \rho )^2 & = \mu \sigma \rho \mu \sigma \rho = \mu \sigma
    \mu^{- 1} \sigma \\
    & = D \left( p_0, \ldots, p_{n + 1} \right) \cdot D \left( - p_{\sigma
    ( 0 )}, \ldots, - p_{\sigma ( n + 1 )} \right) \cdot \sigma^2
    \\
    & = D \left( p_0 - p_{\sigma ( 0 )}, \ldots, p_{n + 1} - p_{\sigma
    ( n + 1 )} \right) \cdot \sigma^2
  \end{align*}
  Therefore, $( \mu \sigma \rho )^2 = \id$ is equivalent to $( \mu \sigma
  \rho )^2 = D \left( N, \ldots, N \right)$ for some $N \in \Z$. 

  When
  $\sigma ( i ) = i$, as $p_i - p_{\sigma ( i )} = 0$, $( \mu \sigma
  \rho )^2 = \id$ is equivalent
  to $\sigma^2 = \id$ and $p_j - p_{\sigma ( j )} \equiv 0 \pmod{d}$ for
  any $j$. 

  When $\sigma$ has no fixed point, $( \mu \sigma \rho )^2 = \id$ is
  equivalent to $\sigma^2 = \id$
  and $p_0 - p_{\sigma ( 0 )} \equiv \cdots \equiv p_{n + 1} - p_{\sigma
  ( n + 1 )} \pmod{d}$. In this case, for $0 \le i \le n + 1$, \[
    p_i - p_{\sigma ( i )} \equiv p_{\sigma ( i )} - p_{\sigma^2 ( i )}
    \equiv p_{\sigma ( i )} - p_i \pmod{d}
  \]
  as $\sigma^2 = \id$. Therefore, in this case, $p_i - p_{\sigma ( i )}
  \equiv 0, \frac{d}{2}$ as $\sigma ( i ) \neq i$, and this is exactly
  what we want to show. 

  (\ref{fh:cocycle_equiv}) Suppose $\beta = \lambda \tau$ be a cocycle
  and $\phi = \nu \psi \in \Aut F^n_d$ with $\lambda, \nu = D \left( q_0,
  \ldots, q_{n + 1} \right) \in G^n_d,\ \tau, \psi \in
  \mathfrak{S}_{n + 2}$.
  Then, 
  \begin{align*}
    & \phi^{- 1} \cdot \beta \rho \cdot \phi \\
    = & \left( \psi^{- 1} \nu^{- 1} \right) \cdot \left( \lambda \tau
    \rho \right) \cdot \left( \nu \psi \right) \\
    = & \psi^{- 1} \cdot \left( \nu^{- 1} \lambda \tau \nu^{- 1} \right)
    \cdot \psi \rho \\
    = & \psi^{- 1} \cdot \left( \nu^{- 1} \lambda \cdot D \left( - q_{\tau
    ( 0 )}, \ldots, - q_{\tau ( n + 1 )} \right) \tau \right) \cdot
    \psi \rho \\
    = & \psi^{- 1} \cdot \left( D \left( q_0 + q_{\tau ( 0 )}, \ldots,
    q_{n + 1} + q_{\tau ( n + 1 )} \right)^{- 1} \lambda \tau \right)
    \cdot \psi \rho
  \end{align*}
  As $q_i$'s vary for all integers, the result follows. 
\end{proof}

\begin{remark}
  By Proposition \ref{fh}, any cocycles in $\Aut F^n_d$ are equivalent
  to one of the following forms: \[
    \begin{pmatrix}
      \xi^{a_0} & & \\
      & \ddots & \\
      & & \xi^{a_{n + 1}}
    \end{pmatrix}
    \cdot
    \begin{pmatrix}
      I_r & & & & & \\
      & 0 & 1 & & & \\
      & 1 & 0 & & & \\
      & & & \ddots & & \\
      & & & & 0 & 1 \\
      & & & & 1 & 0
    \end{pmatrix}
  \]
  with $I_r$ the identity matrix of size $r$, $a_i \in \Z$, and $0 \le r
  \le n + 2$. Throughtout this paper, blank entries in matrices mean $0$.
  Also, as cocycles, \[
    I_r = \xi^m I_r
  \]
  for $m \in \Z$ by definition, and in $\Aut \left( F_d^n \right)$ \[
    \begin{pmatrix}
      1 & \\ & A
    \end{pmatrix} \sim
    \begin{pmatrix}
      \xi^{2 i} & \\ & A
    \end{pmatrix} ,\
    \begin{pmatrix}
      B & & \\ & 0 & 1 \\ & 1 & 0
    \end{pmatrix} \sim
    \begin{pmatrix}
      B & & \\ & 0 & \xi^j \\ & \xi^j & 0
    \end{pmatrix}
  \]
  for $i, j \in \Z$ and any matrices $A, B$ by Proposition \ref{fh}. 
  \label{equiv_comp}
\end{remark}

Now, we can describe $H^1 ( \Gal ( \C / \R ), \Aut F^n_d )$. 

\begin{thm}
  Suppose $d$ is odd. Then, any cocycles in $\Aut F^n_d$ are equivalent
  to exactly one of the following forms:
   \[
    H^r_n :=
    \begin{pmatrix}
      I_r & & & & & \\ & 0 & 1 & & & \\ & 1 & 0 & & & \\ & & & \ddots
      & & \\ & & & & 0 & 1 \\ & & & & 1 & 0
    \end{pmatrix}
    \ ( r \in \Z,\ 0 \le r \le n + 2 ). 
  \]
  In particular, $F^n_d$ has exactly $\left\lfloor \frac{n}{2}
  \right\rfloor + 2$ real forms. 
  \label{odd_d}
\end{thm}

In particular, if $d \ge 3$ is odd, then $F^1_d$ ($d \neq 3$) has
exactly $2$ real forms and $F^2_d$ has $3$ real forms.

\begin{proof}
  As $d$ is odd, \[
    \begin{pmatrix}
      1 & \\ & A
    \end{pmatrix} \sim
    \begin{pmatrix}
      \xi^i & \\ & A
    \end{pmatrix}
  \]
  for all $i$ and matrices $A$. Therefore, when $d$ is odd, cocycles
  $\mu \sigma$,
  $\lambda \tau$ with $\mu, \lambda \in G_d^n$ and $\sigma, \tau \in
  \mathfrak{S}_{n + 2}$ are equivalent if and only if the types of
  $\sigma$ and $\tau$ agree by Proposition \ref{fh} and Remark
  \ref{equiv_comp}.  Therefore, \[
    H^1 ( \Gal ( \C / \R ), \Aut F^n_d ) = \left\{ H^r_n | 0 \le r \le
    n + 2 \right\}
  \]
  where we identify cocycles with their equivalent classes, and
  the statement follows. 
\end{proof}

\begin{thm}
  Suppose $d$ is even. Then, any cocycles in $\Aut F^n_d$ are equivalent
  to exactly one of the forms: \[
    K^{s, t}_n :=
    \begin{pmatrix}
      \xi I_s & & & & & & \\ & I_t & & & & & \\ & & 0 & 1 & & & \\
      & & 1 & 0 & & & \\ & & & & \ddots & & \\ & & & & & 0 & 1 \\
      & & & & & 1 & 0
    \end{pmatrix}
    \ ( s, t \in \Z_{\ge 0},\ s + t \le n + 2 )
  \]
  \[
    L_n :=
    \begin{pmatrix}
      0 & 1 & & & \\
      - 1 & 0 & & & \\
      & & \ddots & & \\
      & & & 0 & 1 \\
      & & & - 1 & 0
    \end{pmatrix}
    \ ( n \equiv 0 \pmod{2} )
  \]
  Also, $F^n_d$ has exactly $\frac{( n + 3 ) ( n + 5 )}{8}$ real forms
  for $n$ odd, and $\frac{( n + 4 ) ( n + 6 )}{8} + 1$ real forms for $n$
  even.
  \label{even_d}
\end{thm}

In particular, if $d \ge 3$ is even, then $F^1_d$ has exactly $3$
real forms and $F^2_d$ ($d \not= 4$) has $7$ real forms. 

\begin{proof}
  By Proposition \ref{fh} and Remark \ref{equiv_comp}, any cocycles
  are equivalent to $L_n$ or \[
    d_r ( a_1, \ldots, a_r ) :=
    \begin{pmatrix}
      \xi^{a_1} & & & & & & & \\
      & \ddots & & & & & & \\
      & & \xi^{a_r} & & & & & \\
      & & & 0 & 1 & & & \\
      & & & 1 & 0 & & & \\
      & & & & & \ddots & & \\
      & & & & & & 0 & 1 \\
      & & & & & & 1 & 0
    \end{pmatrix}
  \]
  with $a_i = 0, 1$ for each $i$. Then, by Proposition \ref{fh},
  $\alpha := d_r ( a_1, \ldots, a_r )$ is equivalent to $\beta :=
  d_r ( b_1, \ldots, b_r )$ if and only if $\beta \sigma^{- 1}
  \alpha^{- 1} \sigma = d_r ( c_1, \ldots, c_r )$ with $\sigma \in
  \mathfrak{S}_{n + 2}$ and $c_i \equiv c_j \pmod{2}$ for all $i, j$,
  which is equivalent to say that $n ( \alpha ) = n ( \beta )$ or $n (
  \alpha ) + n ( \beta ) = r$ where $n ( \alpha ) = \# \left\{ i |
  a_i = 1 \right\}$. Also, $d_r ( a_1, \ldots, a_r )$ and $d_s ( b_1,
  \ldots, b_s )$ are not equivalent if $r \neq s$ by Propositon \ref{fh}.
  Therefore, \[
    H^1 ( \Gal ( \C / \R ), \Aut F^n_d ) \setminus \{ L_n \} =
    \left\{ K^{s, t}_n | s + t \le n + 2,\ 0 \le s \le t \right\}
  \]
  where we identify cocycles with their equivalent classes. 
  From this, when $n = 2 k - 1$ for $k \in \Z$, the number of real forms
  is \[
    \sum_{i = 0}^k ( i + 1 ) = \frac{( k + 1 ) ( k + 2 )}{2} = \frac{(
    n + 3 ) ( n + 5 )}{8}, 
  \]
  and when $n = 2 k$ for $k \in \Z$, the number of real forms is \[
    1 + \sum_{i = 0}^{k + 1} ( i + 1 ) = 1 + \frac{( k + 2 ) ( k + 3 )}{2} =
    \frac{( n + 4 ) ( n + 6 )}{8} + 1,
  \]
  so the result follows. 
\end{proof}

\section{The elliptic curve case}
\label{num_ell}

In this section, we consider the real forms of $F^1_3 = ( X_0^3 + X_1^3
+ X_2^3 = 0 ) \subset \P^2$. First, the substitution \[
  X = \frac{-12 X_2}{X_0 + X_1},\ Y = \frac{36 ( X_0 - X_1 )}{X_0 + X_1}
\]
gives an elliptic curve $Y^2 = X^3 - 432$ with $j$-invariant $0$, and
this can be identified with $E := \C / ( \Z + \zeta \Z )$ where $\zeta$
is a primitive $6$-th root of unity, which we fix in this section. 

Then, any automorphisms of $E$ can be expressed by \[
  z \mapsto \zeta^i z + a
\]
with $i \in \Z$ and $a$ in the fundamental domain (cf. \cite{Sil09}). 
Note that the real structure $\rho \colon F^1_3 \to F^1_3$ corresponds
to $E \to E$ that maps $z$ to $\overline{z}$ where we identify $z \in \C$ with
the image of $z$ by the canonical surjection $\C \to \C / ( \Z + \zeta
\Z )$. We also denote it by $\rho \colon E \to E$. 

\begin{prop}
  We fix $\xi \in \C$ so that $\xi^2 = \zeta$.
  \begin{enumerate}
    \item $\alpha \colon E \to E;\ z \mapsto \zeta^i z + a$ is a cocycle
      if and only if there exists $r \in \R$ such that $a = \sqrt{- 1}
      \xi^i r$.
      \label{ell:cocycle}
    \item Any cocycles are equivalent to $z \mapsto z$ or $z \mapsto
      - z$, which are not equivalent to each other. 
      \label{ell:cocycle_equiv}
  \end{enumerate}
\end{prop}

\begin{proof}
  (\ref{ell:cocycle}) By definition, $\alpha$ is a cocycle if and only
  if $( \alpha \circ \rho )^2 = \id$. 
  \begin{align*}
    ( \alpha \circ \rho )^2 ( z ) & = \zeta^i ( \overline{\zeta^i
    \overline{z} + a} ) + a \\
    & = z + ( \zeta^i \overline{a} + a )
  \end{align*}
  Therefore, the result follows. 

  (\ref{ell:cocycle_equiv}) Let $\sigma \colon E \to E;\ z \mapsto
  \zeta^i z + a$ be a cocycle and $\tau \colon E \to E;\ z \mapsto
  \zeta^j z + b$. Then, $\tau^{- 1} \colon z \mapsto \zeta^{- j} z -
  \zeta^j b$. Therefore, 
  \begin{align*}
    & ( \tau^{- 1} \circ ( \sigma \rho ) \circ \tau ) ( z ) \\
    = & \tau^{- 1} ( \zeta^i ( \overline{\zeta^j z + b} ) + a ) \\
    = & \zeta^{i - 2 j} \overline{z} + \zeta^{- j} a - \zeta^j b +
    \zeta^{i - j} \overline{b}
  \end{align*}
  First, $z \mapsto z$ and $z \mapsto - z$ are not equivalent to each
  other by this calculation.

  Also, in the above, we can write $a = \sqrt{- 1} \xi^i x$ for
  $x \in \R$ by (\ref{ell:cocycle}). Then, by setting \[
    j = 0,\ b = \frac{1 + \sqrt{- 1}}{2} \xi^i x,
  \]
  we see that $z \to \zeta^i z + a$ is equivalent to $z \to \zeta^i z$,
  so, by setting $a = b = 0$ in the above calculation, we get the result.
\end{proof}

Note that the negation map in the above corresponds to a linear
automorphism \[
  F^1_3 \to F^1_3;\ [ X_0 : X_1 : X_2 ] \mapsto [ X_1 : X_0 : X_2 ]
\]
by taking $[ 1 : - 1 : 0 ]$ as a unit. From this, we immediately get the
following result:

\begin{thm}
  In $F^1_3$, any cocycles are equivalent to $H^3_1$ or $H^1_1$. 

  In particular, $F^1_3$ has exactly $2$ real forms. 
  \label{ell}
\end{thm}

\section{Degree $2$ cases}
\label{quad}

(This section is based on the advice from Professor J\'{a}nos
Koll\'{a}r. )

In this section, we consider the number and explicit descriptions of the
real forms of $X_n := F^n_2 =
\left( X_0^2 + \cdots + X_{n + 1}^2 = 0 \right) \subset \P^{n + 1}$. 

Let $Q_{r, s} \subset \P^{r + s - 1}$ be a quadratic hypersurface
defined as follows: \[
  Q_{r, s} := \left( X_0^2 + \cdots + X_{r - 1}^2 - X_r^2 - \cdots
  X_{r + s - 1}^2 = 0 \right) \subset \P^{r + s - 1}.
\]

Then, $F^1_2$ is isomorphic to $\P^1$, and has exactly two real forms
$Q_{2, 0}$ and $Q_{1, 1}$ up to isomorphism. Also, it is known that
$F^2_2$ has exactly four real forms $Q_{4, 0}$, $Q_{3, 1}$, $Q_{2, 2}$,
and $Q_{3, 0} \times Q_{2, 1}$ up to isomorphisms (cf. \cite{Rus02}). 

Therefore, we consider the case where $n \ge 3$. 

\begin{lem}
  When $n \ge 3$, any real forms of $X_n$ can be embedded into real forms
  of $\P^{n + 1}$.
  \label{qemb}
\end{lem}

\begin{proof}
  Let $H$ be a hyperplane section of $X_n$. Then, $| H |$ induces an
  embedding $\varphi_{| H |} \colon X_n \hookrightarrow \P^{n + 1}$. 

  Considering the automorphism group of $X_n$, we see that $X_n$ is also
  a smooth hypersurface of degree $2$ in $\P^{n + 1}$ in this embedding. 
  Then, every real structure of $X_n$ preserves $| H |$ as every
  automorphism of $X_n$ is linear. Therefore, by considering the quotient
  by a real structure, we see that every real form of $X_n$ can be
  embedded into some real form of $\P^{n + 1}$. 
\end{proof}

\begin{thm}
  When $n$ is odd, any real forms of $F^2_n$ are isomorphic to $Q_{r, s}$
  for some $r$,~$s$ with $r \ge s, r + s = n + 2$. 
  When $n$ is even, any real forms of $F^2_n$ are isomorphic to
  $Q_{r, s}$ for some $r$,~$s$ with $r \ge s, r + s = n + 2$, or
  the real form corresponding to the cocycle $L_n$ in $\Aut X_n$. 

  In particular, when $n$ is even, $F^2_n$ has exactly $\frac{n}{2} + 3$
  real forms, and when $n$ is odd, $F^2_n$ has exactly $\frac{n + 1}{2}
  + 1$ real forms.
  \label{qrf}
\end{thm}

\begin{proof}
  By \cite{Ben16}, $P^{n + 1}$ has exactly two real forms up to
  isomorphisms when $n$ is even,
  which correspond to the following two real structures: 
  \begin{align*}
    \tilde{\rho} & \colon \P^{n + 1} \to \P^{n + 1};\ [ z_0 : \cdots :
    z_{n + 1} ] \mapsto \left[ \overline{z_0} : \cdots :
    \overline{z_{n + 1}} \right] \\
    \tilde{\rho'} & \colon \P^{n + 1} \to \P^{n + 1};\ [ z_0 : \cdots :
    z_{n + 1} ] \mapsto \left[ - \overline{z_1} : \overline{z_0} :
    \cdots : \overline{z_{n + 1}} : - \overline{z_n} \right].
  \end{align*}
  Here, we denote by $\widetilde{\P^{n + 1}}$ the real form of
  $\P^{n + 1}$
  defined by the real structure $\tilde{\rho'}$.
  Also, when $n$ is odd, $\P^{n + 1}$ has exactly one real form up to
  isomorphisms, which corresponds to above $\tilde{\rho}$. 

  Therefore, by Lemma \ref{qemb}, we enough to consider the case where
  a real form of $F^2_n$ is embedded into these two real forms of
  $\P^{n + 1}$. 

  First, suppose a real form $Y$ of $X_n$ is embedded into the real form
  of $\P^{n + 1}$ corresponding to the real strucure $\tilde{\rho}$,
  i.e., $\P^{n + 1}_\R$. Then, $Y$
  is a hypersurface of degree $2$ of $\P^{n + 1}_\R$ by the proof of
  Lemma \ref{qemb}. As $F^n_2$ is defined by a non-degenerated quadratic
  form, $Y$ is isomorphic to $Q_{r, s}$ for some $r, s$ with $r \ge s$
  by considering a diagonalization of symmetric matrix. 

  Second, suppose a real form $Y$ of $X_n$ is embedded into the real form
  of $\P^{n + 1}$ corresponding to the real strucure $\tilde{\rho'}$,
  i.e., $\widetilde{\P^{n + 1}}$. Then, $\widetilde{\P^{n + 1}}$ contains
  $\widetilde{\P^1}$, and as $F^n_2$ is defined by a non-degenerated
  quadratic form, we can take $L \subset \widetilde{\P^{n + 1}}$ which is
  isomorphic to $\widetilde{\P^1}$ so that it is not contained in $Y$.
  We may assume that $L = \left( X_2 = \cdots = X_{n + 1} = 0 \right)
  \subset \widetilde{\P^{n + 1}}$. Also, we may assume that $Y \cap L =
  \left( X_0^2 + X_1^2 = 0 \right) \subset L$ as hypersurfaces in
  $\widetilde{\P^1}$ defined by $\left( X_0^2 + X_1^2 = 0 \right)$ and
  $\left( X_0^2 - X_1^2 = 0 \right)$ is isomorphic. Then, by considering
  orthogonal complement, we can set the
  coordinate of $\widetilde{\P^{n + 1}}$ so that $Y$ is defined by \[
    X_0^2 + X_1^2 + q \left( X_2, \ldots, X_{n + 1} \right) = 0
  \]
  where $q$ is a non-degenerated quadratic form. Therefore, by recurring
  this procedure, we get the unique real form \[
    \left( X_0^2 + \cdots X_{n + 1}^2 = 0 \right) \subset
    \widetilde{\P^{n + 1}}
  \]
  of $X_n$ embedded in $\widetilde{\P^{n + 1}}$, and it is clear that this
  real form correspond the cocycle $L_n$ in $\Aut X_n$. 

  By considering the real loci of $Q_{r, s}$, we see that $Q_{r, s}$'s
  with $r \ge s$ are not isomorphic to each other. Also, a real form
  $Q_{r, s}$ of $X_n$ corresponds to a cocycle $K^{s, r}_2$, and when
  $n$ is even, this cannot be equivalent to $L_n$. In fact, as matrices,
  \[
    \begin{pmatrix}
      - I_s & \\
      & I_r
    \end{pmatrix}^2
    = I_{r + s}
  \]
  whereas \[
    \begin{pmatrix}
      0 & 1 & & & \\
      - 1 & 0 & & & \\
      & & \ddots & & \\
      & & & 0 & 1 \\
      & & & - 1 & 0
    \end{pmatrix}^2
    = - I_{n + 2}.
  \]
  Therefore, we see that the number of real forms of $X_n$ is
  $\frac{n}{2} + 3$ when $n$ is even, and $\frac{n + 1}{2} + 1$ when $n$
  is odd. 
\end{proof}

\section{Explicit description of real forms of Fermat hypersurfaces}

In this section, we consider the explicit description of real forms
corresponding to $H^r_n$, $K^{s, t}_n$, and $L_n$. 
We omit subindexes of $H^r_n$ or $K^{s, t}_n$ if it is clear from the
context. 

First, we consider the real form corresponding to cocycles $L_n$. As we
see in Theorem \ref{odd_d}, Theorem \ref{even_d}, and Theorem
\ref{ell}, cocycles of this type appear only when $n$ and $d$ are
even. Let $n = 2 k$. 
This is a real form induced by the real form $\widetilde{\P^{n + 1}}$ of
$\P^{n + 1}$, i.e., corresponding to the real structure $\tilde{\rho'}$.

Therefore, the quotient $Z^{2 k}_d$ of $F_{d, \C}^{2 k} \subset
\P_\C^{2 k + 1}$ with even $d$ by the restriction of $\rho'$ is the real
form corresponding to $L_{2 k}$. Note that this real form can be embedded
into $\widetilde{\P^{n + 1}}$. 

Next, we consider the real form corresponding to cocycles $K$ or $H$.
When considering these real forms, the following lemma is useful: 

\begin{lem}
  Let $\xi_d$ be a primitive $d$-th root of unity. Then, 
  \begin{enumerate}
    \item The invariant ring of the $\R$-homomorphism $\alpha_1 \colon
      \C [ X ] \to \C [ X ]$ that sends $X$ to $\xi_d X$ and $c \in \C$
      to $\overline{c}$ is $\R \left[ \xi_{2 d} X \right]$, with
      $\xi_{2 d}$ satisfying $\xi_{2 d}^2 = \xi_d$. 
      \label{inv1}
    \item The invariant ring of the $\R$-homomorphism $\alpha_2 \colon
      \C [ X_1, X_2 ] \to \C [ X_1, X_2 ]$ that sends $X_1$ to $X_2$,
      $X_2$ to $X_1$, and $c \in \C$ to $\overline{c}$ is $\R \left[ X_1
      + X_2,
      \left( 1 - \sqrt{- 1} \right) \left( X_1 - \sqrt{- 1} X_2 \right)
      \right]$. 
      \label{inv2}
  \end{enumerate}
  \label{inv_ring}
\end{lem}

\begin{proof}
  (\ref{inv1}) For all $f \in \C [ X ]$, there is a unique polynomial
  $\mu ( f ) \in \C [ X ]$ such that \[
    f ( X ) = \mu \left( \xi_{2 d} X \right). 
  \]
  Then, as $\xi_{2 d} X$ is invariant under $\alpha_1$, we see that
  $\alpha_1 \left( f \right) = f$ if and only if $\mu ( f )$ is a
  polynomial of real coefficients.  This shows that the invariant ring
  with $\alpha_1$ is $\R \left[ \xi_{2 d} \right]$. 

  (\ref{inv2}) For all $g \in \C [ X_1, X_2 ]$, there is a unique
  polynomial $\nu ( g ) \in \C [ X_1, X_2 ]$ such that \[
    g ( X_1, X_2 ) = \nu ( g ) \left( X_1 + X_2, \left( 1 - \sqrt{- 1}
    \right) \left( X_1 - \sqrt{- 1} X_2 \right) \right).
  \]
  Then, as $X_1 + X_2$, $\left( 1 - \sqrt{- 1} \right) \left( X_1 -
  \sqrt{- 1} X_2 \right)$ is invariant under $\alpha_2$, we see that
  $\alpha_2 ( g ) = g$ if and only if $\nu ( g )$ is a polynomial with
  real coefficients. 
  This shows that the invariant ring of $\alpha_2$ is $\R \left[ X_1 +
  X_2, \left( 1 - \sqrt{- 1} \right) \left( X_1 - \sqrt{- 1} X_2 \right)
  \right]$. 
\end{proof}

Then, we consider the curve case $F^1_d = ( X_0^d + X_1^d + X_2^d
= 0 )$. 

\begin{thm}
  Assume that $d \ge 3$. Then, 
  the correspondence between cocycles and real forms of $F^1_d$ is given
  in the following table. 
  %\begin{enumerate}
  %  \item the real form of $F^1_d$ corresponding to the cocycle
  %    $K^{0, 3}$ (or $H^0$) is \[
  %      ( Y_0^d + Y_1^d + Y_2^d = 0 ) \subset \P_\R^2. 
  %    \]
  %    \label{rf:03}
  %  \item the real form of $F^1_d$ corresponding to the cocycle
  %    $K^{1, 2}$ is \[
  %      ( - Y_0^d + Y_1^d + Y_2^d = 0 ) \subset \P_\R^2. 
  %    \]
  %    \label{rf:12}
  %  \item the real form of $F^1_d$ corresponding to the cocycle
  %    $K^{0, 1}$ (or $H^2$) is \[
  %      ( Y_0^d + ( Y_1 + \sqrt{- 1} Y_2 )^d + ( Y_1 -
  %      \sqrt{- 1} Y_2 )^d = 0 ) \subset \P_\R^2. 
  %    \]
  %    \label{rf:01}
  %\end{enumerate}

  \begin{center}
    \begin{tabular}{|c|c|} \hline
      cocycle & real form \\ \hline
      $K^{0, 3}_1$ (or $H^0_1$) & $\displaystyle ( Y_0^d + Y_1^d + Y_2^d
      = 0 ) \subset \P^2_\R$ \\ \hline
      $K^{1, 2}_1$ & $\displaystyle ( - Y_0^d + Y_1^d + Y_2^d = 0 )
      \subset \P^2_\R$ \\ \hline
      $K^{0, 1}_1$ (or $H^2_1$) & $\displaystyle ( Y_0^d + ( Y_1 +
      \sqrt{- 1} Y_2 )^d + ( Y_1 - \sqrt{- 1} Y_2 )^d = 0 ) \subset
      \P_\R^2$ \\
      \hline
    \end{tabular}
  \end{center}

  \label{des:curve}
\end{thm}

\begin{proof}
  %(\ref{rf:03})
  First, the real form corresponding to $K^{0, 3}$ is clear. 

  %(\ref{rf:12})
  Next, we consider the real form corresponding to $K^{1, 2}$. 
  By Lemma \ref{inv_ring}, the invariant ring of the homomorphism
  $\alpha \colon \C [ X_0, X_1, X_2 ] \to \C [ X_0, X_1, X_2 ]$ that
  sends $c \in \C$ to $\overline{c}$, $X_0$ to $\xi_d X_0$, $X_1$ to $X_1$,
  and $X_2$ to $X_2$ is $\R [ \xi_{2 d} X_0, X_1, X_2 ]$ with
  $\xi_{2 d}^2 = \xi_d$.

  Therefore, by setting \[
    Y_0 := \xi_{2 d} X_0,\ Y_1 := X_1,\ Y_2 := X_2,
  \]
  we see that the corresponding real form is \[
    ( - Y_0^d + Y_1^d + Y_2^d = 0 ) \subset \P_\R^2. 
  \]

  %(\ref{rf:01})
  Finally, we consider the real form corresponding to $K^{1, 2}$.
  By Lemma \ref{inv_ring}, the invariant ring of the homomorphism $\beta
  \colon \C [ X_0, X_1, X_2 ] \to \C [ X_0, X_1, X_2 ]$ that sends $c \in
  \C$ to $\overline{c}$, $X_0$ to $X_0$, $X_1$ to $X_2$ and $X_2$ to $X_1$ is
  $\R [ X_0, X_1 + X_2, ( 1 - \sqrt{- 1} ) ( X_1 - \sqrt{- 1} X_2 ) ]$.

  Therefore, by setting \[
    Y_0 := X_0, \ {Y'}_1 := X_1 + X_2,\ {Y'}_2 := ( 1 - \sqrt{- 1} )
    ( X_1 - \sqrt{- 1} X_2 ), 
  \]
  and also, by setting \[
    Y_1 := \frac{{Y'}_1}{2},\ Y_2 := \frac{{Y'}_2 - {Y'}_1}{2}, 
  \]
  we see that the corresponding real form is \[
    ( Y_0^d + ( Y_1 + \sqrt{- 1} Y_2 )^d + ( Y_1 - \sqrt{- 1} Y_2 )^d
    = 0 ) \subset \P_\R^2. 
  \]
\end{proof}

In the surface case, by using a similar argument as above, we get the
following result: 

\begin{thm}
  Assume $d \ge 3, \neq 4$. Then, 
  the correspondence between cocycles and real forms of $F^2_d$ with odd
  $d$ is given in the following table: 
  %\begin{enumerate}
  %  \item the real form of $F^2_d$ corresponding to the cocycle
  %    $K^{0, 4}_2$ (or $H^4$) is \[
  %      ( Y_0^d + Y_1^d + Y_2^d + Y_3^d = 0 ) \subset \P_\R^3. 
  %    \]
  %  \item the real form of $F^2_d$ corresponding to the cocycle
  %    $K^{1, 3}_2$ is \[
  %      ( - Y_0^d + Y_1^d + Y_2^d + Y_3^d = 0 ) \subset \P_\R^3.
  %    \]
  %  \item the real form of $F^2_d$ corresponding to the cocycle
  %    $K^{2, 2}_2$ is \[
  %      ( - Y_0^d - Y_1^d + Y_2^d + Y_3^d = 0 ) \subset \P_\R^3.
  %    \]
  %  \item the real form of $F^2_d$ corresponding to the cocycle
  %    $K^{0, 2}_2$ is \[
  %      ( Y_0^d + Y_1^d + ( Y_2 + \sqrt{- 1} Y_3 )^d + ( Y_2 - \sqrt{- 1}
  %      Y_3 )^d = 0 ) \subset \P_\R^3. 
  %    \]
  %  \item the real form of $F^2_d$ corresponding to the cocycle
  %    $K^{1, 1}_2$ is \[
  %      ( - Y_0^d + Y_1^d + ( Y_2 + \sqrt{- 1} Y_3 )^d + ( Y_2 -
  %      \sqrt{- 1} Y_3 )^d = 0 ) \subset \P_\R^3. 
  %    \]
  %  \item the real form of $F^2_d$ corresponding to the cocycle
  %    $K^{0, 0}_2$ is \[
  %      ( ( Y_0 + \sqrt{- 1} Y_1 )^d + ( Y_0 - \sqrt{- 1} Y_1 )^d +
  %      ( Y_2 + \sqrt{- 1} Y_3 )^d + ( Y_2 - \sqrt{- 1} Y_3 )^d = 0 )
  %      \subset \P_\R^3. 
  %    \]
  %\end{enumerate}

  \begin{center}
    \begin{tabular}{|c|c|} \hline
      cocycle & real form \\ \hline
      $H^4_2$ & $\displaystyle
      ( Y_0^d + Y_1^d + Y_2^d + Y_3^d = 0 ) \subset \P_\R^3$
      \\ \hline
      $H^2_2$ & $\displaystyle
      ( Y_0^d + Y_1^d + ( Y_2 + \sqrt{- 1} Y_3 )^d + ( Y_2 - \sqrt{- 1}
      Y_3 )^d = 0 ) \subset \P_\R^3$
      \\ \hline
      $H^0_4$ & $\displaystyle
      ( ( Y_0 + \sqrt{- 1} Y_1 )^d + ( Y_0 - \sqrt{- 1} Y_1 )^d +
      ( Y_2 + \sqrt{- 1} Y_3 )^d + ( Y_2 - \sqrt{- 1} Y_3 )^d = 0 )
      \subset \P_\R^3$
      \\ \hline
    \end{tabular}
  \end{center}

  Also, the correspondence between cocycles and real forms of $F^2_d$
  with even $d$ is given in the following table: 

  \begin{center}
    \begin{tabular}{|c|c|} \hline
      cocycle & real form \\ \hline
      $K^{0, 4}_2$ & $\displaystyle
      ( Y_0^d + Y_1^d + Y_2^d + Y_3^d = 0 ) \subset \P_\R^3$
      \\ \hline
      $K^{1, 3}_2$ & $\displaystyle
      ( - Y_0^d + Y_1^d + Y_2^d + Y_3^d = 0 ) \subset \P_\R^3$
      \\ \hline
      $K^{2, 2}_2$ & $\displaystyle
      ( - Y_0^d - Y_1^d + Y_2^d + Y_3^d = 0 ) \subset \P_\R^3$
      \\ \hline
      $K^{0, 2}_2$ & $\displaystyle
      ( Y_0^d + Y_1^d + ( Y_2 + \sqrt{- 1} Y_3 )^d + ( Y_2 - \sqrt{- 1}
      Y_3 )^d = 0 ) \subset \P_\R^3$
      \\ \hline
      $K^{1, 1}_2$ & $\displaystyle
      ( - Y_0^d + Y_1^d + ( Y_2 + \sqrt{- 1} Y_3 )^d + ( Y_2 -
      \sqrt{- 1} Y_3 )^d = 0 ) \subset \P_\R^3$
      \\ \hline
      $K^{0, 0}_2$ & $\displaystyle
      ( ( Y_0 + \sqrt{- 1} Y_1 )^d + ( Y_0 - \sqrt{- 1} Y_1 )^d +
      ( Y_2 + \sqrt{- 1} Y_3 )^d + ( Y_2 - \sqrt{- 1} Y_3 )^d = 0 )
      \subset \P_\R^3$
      \\ \hline
      $L_2$ & $Z_d^2$
      \\ \hline
    \end{tabular}
  \end{center}

  \label{des:surface}
\end{thm}

\begin{proof}
  The real form corresponding to $L_2$ is already described above. 
  First, we consider the real form corresponding to $K_2^{k, l}$ with
  $k + l = 4$. 
  The real form corresponding to $K_2^{0, 4}$ is clear. 

  By Lemma \ref{inv_ring}, the invariant ring of the $\R$-homomorphism
  $\alpha_1 \colon \C [ X_0, X_1, X_2, X_3 ] \to \C [ X_0, X_1, X_2, X_3]$
  that sends $X_0$ to $\xi_d X_0$, $X_i$ to $X_i$ for $i = 1, 2, 3$, and
  $c \in \C$ to $\overline{c}$
  is $\R [ \xi_{2 d} X_0, X_1, X_2, X_3 ]$. Therefore, by setting \[
    Y_0 := \xi_{2 d} X_0,\ Y_i := X_i \ ( i = 1, 2, 3 ),
  \]
  we see that the real form corresponding to $K_2^{1, 3}$ is \[
    \left( - Y_0^d + Y_1^d + Y_2^d + Y_3^d = 0 \right) \subset \P_\R^3.
  \]

  Also, by Lemma \ref{inv_ring}, the invariant ring of the
  $\R$-homomorphism
  $\alpha_2 \colon \C [ X_0, X_1, X_2, X_3 ] \to \C [ X_0, X_1, X_2, X_3]$
  that sends $X_i$ to $\xi_d X_i$ for $i = 0, 1$, $X_j$ to $X_j$ for
  $j = 2, 3$, and $c \in \C$ to $\overline{c}$
  is $\R [ \xi_{2 d} X_0, \xi_{2 d} X_1, X_2, X_3 ]$. Therefore, by
  setting \[
    Y_i := \xi_{2 d} X_i\ ( i = 0, 1),\ Y_j := X_j \ ( j = 2, 3 ),
  \]
  we see that the real form corresponding to $K_2^{2, 2}$ is \[
    \left( - Y_0^d - Y_1^d + Y_2^d + Y_3^d = 0 \right) \subset \P_\R^3.
  \]

  Next, we consider the real form corresponding to $K_2^{k, l}$ with
  $k + l = 2$. 

  By Lemma \ref{inv_ring}, the invariant ring of the $\R$-homomorphism
  $\beta_1 \colon \C [ X_0, X_1, X_2, X_3 ] \to \C [ X_0, X_1, X_2, X_3]$
  that sends $X_i$ to $X_i$ for $i = 0, 1$, $X_2$ to $X_3$, $X_3$ to
  $X_2$, and $c \in \C$ to $\overline{c}$
  is $\R [ X_0, X_1, X_2 + X_3, \left( 1 - \sqrt{- 1} \right) \left( X_2
  - \sqrt{- 1} X_3 \right) ]$. Therefore, by setting \[
    Y_i := X_i \ ( i = 0, 1 ),\ {Y'}_2 := X_2 + X_3, {Y'}_3 := \left(
    1 - \sqrt{- 1} \right) \left( X_2 - \sqrt{- 1} X_3 \right)
  \]
  and also, by setting \[
    Y_2 := \frac{{Y'}_2}{2},\ Y_3 := \frac{{Y'}_3 - {Y'}_2}{2}, 
  \]
  we see that the real form corresponding to $K_2^{1, 3}$ is \[
    \left( Y_0^d + Y_1^d + \left( Y_2 + \sqrt{- 1} Y_3 \right)^d +
    \left( Y_2 - \sqrt{- 1} Y_3 \right)^d = 0 \right) \subset \P_\R^3.
  \]

  Also, by Lemma \ref{inv_ring}, the invariant ring of the
  $\R$-homomorphism $\beta_2 \colon \C [ X_0, X_1, X_2, X_3 ] \to \C
  [ X_0, X_1, X_2, X_3]$ that sends $X_0$ to $\xi_d X_0$, $X_1$ to $X_1$,
  $X_2$ to $X_3$, $X_3$ to $X_2$, and $c \in \C$ to $\overline{c}$
  is $\R [ \xi_{2 d} X_0, X_1, X_2 + X_3, \left( 1 - \sqrt{- 1} \right)
  \left( X_2 - \sqrt{- 1} X_3 \right) ]$. Therefore, by setting \[
    Y_0 := \xi_{2 d} X_0,\ Y_1 := X_1 ,\ {Y'}_2 := X_2 + X_3, {Y'}_3 :=
    \left( 1 - \sqrt{- 1} \right) \left( X_2 - \sqrt{- 1} X_3 \right)
  \]
  and also, by setting \[
    Y_2 := \frac{{Y'}_2}{2},\ Y_3 := \frac{{Y'}_3 - {Y'}_2}{2}, 
  \]
  we see that the real form corresponding to $K_2^{1, 3}$ is \[
    \left( - Y_0^d + Y_1^d + \left( Y_2 + \sqrt{- 1} Y_3 \right)^d +
    \left( Y_2 - \sqrt{- 1} Y_3 \right)^d = 0 \right) \subset \P_\R^3.
  \]

  Finally, we consider the real form corresponding to $K_2^{0, 0}$. 
  By Lemma \ref{inv_ring}, the invariant ring of the $\R$-homomorphism
  $\gamma \colon \C [ X_0, X_1, X_2, X_3 ] \to \C [ X_0, X_1, X_2, X_3]$
  that sends $X_i$ to $X_{i + 1}$, $X_{i + 1}$ to $X_i$ for $i = 0, 2$,
  and $c \in \C$ to $\overline{c}$
  is $\R [ X_0 + X_1, \left( 1 - \sqrt{- 1} \right) \left( X_0 -
  \sqrt{- 1} X_1 \right), X_2 + X_3, \left( 1 - \sqrt{- 1} \right)
  \left( X_2 - \sqrt{- 1} X_3 \right) ]$. Therefore, by setting \[
    {Y'}_i := X_i + X_{i + 1}, {Y'}_{i + 1} := \left( 1 - \sqrt{- 1}
    \right) \left( X_i - \sqrt{- 1} X_{i + 1} \right) \ ( i = 0, 2 )
  \]
  and also, by setting \[
    Y_i := \frac{{Y'}_i}{2},\ Y_{i + 1} := \frac{{Y'}_{i + 1} -
    {Y'}_i}{2} \ ( i = 0, 2 ),
  \]
  we see that the corresponding real form is \[
    \left( \left( Y_0 + \sqrt{- 1} Y_1 \right)^d + \left( Y_0 - \sqrt{- 1}
    Y_1 \right)^d + \left( Y_2 + \sqrt{- 1} Y_3 \right)^d + \left( Y_2 -
    \sqrt{- 1} Y_3 \right)^d = 0 \right) \subset \P_\R^3.
  \]
\end{proof}

In the same way, we can extend it to the general case. In particular, all
real forms of $F_d^n$ can be described as the $n$-dimensional hypersurface
of degree $d$ except for the real form corresponding to $L_n$.

Next, we will give a topological description of a real locus of these
real forms. 

\begin{thm}
  \begin{enumerate}
    \item The real locus of $Z^{2 k}_d$ with even $d$ is $\emptyset$ for
      all $k$ and $d$.
      \label{rl:l}
    \item The real locus of \[
        ( X_0^d + \cdots + X_{n + 1}^d = 0 )
      \]
      with $d$ even is $\emptyset$. 
      \label{rl:empty}
    \item The real locus of \[
        ( X_0^d + \cdots + X_n^d - X_{n + 1}^d = 0 )
      \]
        with $d$ even is homeomorphic to $S^n$. 
      \label{rl:-1}
    \item The real locus of \[
        ( X_0^d + \cdots + X_i^d - X_{i + 1}^d - \cdots
        - X_{n + 1}^d = 0 )
      \]
      with $d$ even is homeomorphic to $\P^i ( \R ) \times \P^{n - i}
      ( \R )$. 
      \label{rl:-r}
    \item The real locus of
      %\begin{multline*}
      %  ( X_0^d + \cdots + X_{i - 1}^d + \\
      %  ( X_i + \sqrt{- 1} X_{i + 1} )^d
      %  + ( X_i - \sqrt{- 1} X_{i + 1} )^d + \cdots + \\
      %  ( X_n + \sqrt{- 1} X_{n + 1} )^d + ( X_n - \sqrt{- 1}
      %  X_{n + 1} )^d = 0 )
      %\end{multline*}
      \[
        \left( \sum_{k = 0}^{i - 1} X_k^d + \sum_{l = 0}^{\frac{n - i}{2}}
        \left( \left( X_{i + 2 l} + \sqrt{- 1} X_{i + 2 l + 1} \right)^d
        + \left( X_{i + 2 l} - \sqrt{- 1} X_{i + 2 l + 1} \right)^d
        \right) = 0 \right)
      \]
      with $d$ odd and $i > 0$ is homeomorphic to $\P^n ( \R )$. 
      \label{rl:most_odd}
    \item The real locus of \[
        ( X_0^d + \cdots + X_{n - 1}^d + ( X_n + \sqrt{- 1}
        X_{n + 1} )^d + ( X_n - \sqrt{- 1} X_{n + 1} )^d = 0 )
      \]
      with $d$ even is homeomorphic to $\underbrace{S^n \sqcup \cdots
      \sqcup S^n}_{\frac{d}{2}}$. 
      \label{rl:even1}
  \end{enumerate}
  \label{rl:gen}
\end{thm}

\begin{proof}
  In each case, let $Z$ be the real locus. 

  (\ref{rl:l}) As the restriction of $\rho'$ to $F^{2 k}_{d, \C}$ has
  no fixed real points, the real locus of $Z^{2 k}_d$ is $\emptyset$. 

  (\ref{rl:empty}) is clear. 

  (\ref{rl:-1}) Note that $[ x_0 : \cdots : x_{n + 1} ] \in Z$ implies
  $x_{n + 1} \neq 0$. Therefore, we can define a continuous map $Z \to
  S^n$ as follows: \[
    [ x_0 : \cdots : x_n : 1 ] \mapsto \frac{1}{\sqrt{x_0^2 +
    \cdots + x_n^2}} ( x_0, \ldots, x_n ).
  \]
  As it is a bijection, it is also a homeomorphism, and thus,
  $Z \approx S^n$. 

  (\ref{rl:-r}) Note that $[ x_0 : \cdots : x_{n + 1} ] \in Z$ implies
  $( x_0, \ldots, x_i ) \neq ( 0, \ldots, 0 )$, $( x_{i + 1}, \ldots,
  x_{n + 1}) \neq ( 0, \ldots, 0 )$. Thus, we can define a continuous map
  $Z \to \P^i ( \R ) \times \P^{n - i} ( \R )$ as follows: \[
    [ x_0 : \cdots : x_{n + 1} ] \mapsto ( [ x_0 : \cdots : x_i ],
    [ x_{i + 1} : \cdots : x_{n + 1} ] ).
  \]
  As it is a bijection, it is also a homeomorphism, and thus, 
  $Z \approx \P^i ( \R ) \times \P^{n - i} ( \R )$. 

  (\ref{rl:most_odd}) Note that $[ x_0 : \cdots : x_{n + 1} ] \in Z$
  implies $( x_1, \ldots, x_{n + 1} ) \neq ( 0, \ldots, 0 )$. Therefore,
  we can define a continuous function $Z \to \P^n ( \R )$ as follows: \[
    [ x_0 : \cdots : x_{n + 1} ] \mapsto [ x_1 : \cdots : x_{n + 1} ].
  \]
  As $d$ is odd, this is a bijection, and thus, a homeomorphism.
  So $Z \approx \P^n ( \R )$. 

  (\ref{rl:even1}) First, note that $[ x_0 : \cdots : x_{n + 1} ] \in Z$
  implies $( x_n, x_{n + 1} ) \neq ( 0, 0 )$. Therefore, we can define
  a continuous map $Z \to \P^i ( \R )$ as follows: \[
    [ x_0 : \cdots : x_{n + 1} ] \mapsto [ x_n : x_{n + 1} ].
  \]
  Then, this image has $\frac{d}{2}$ connected components \[
    E_m := \left\{ [ \cos \theta : \sin \theta ] \bigg| \frac{( 4 m - 1 )
    \pi}{2 d} \le \theta \le \frac{( 4 m + 1 ) \pi}{2 d} \right\}
  \]
  for $0 \le m \le \frac{d}{2} - 1$. 

  In fact, $[ x_n : x_{n + 1} ] \in f ( Z )$ if and only if
  $( x_n + \sqrt{- 1} x_{n + 1})^d + ( x_n - \sqrt{- 1} x_{n + 1} )^d
  \le 0$, and \[
    ( \cos \theta + \sqrt{- 1} \sin \theta )^d + ( \cos \theta -
    \sqrt{- 1} \sin \theta )^d = 2 \cos d \theta. 
  \]
  Also, for each $m$, we can define a continuous map $f^{- 1} ( E_m )
  \to S^n$ defined as follows: \[
    [ x_0 : \cdots : x_{n - 1} : \cos \theta : \sin \theta ] \mapsto
    \frac{1}{\sqrt{x_0^2 + \cdots x_{n - 1}^2 + \sin d \theta}}
    ( x_0, \ldots, x_{n - 1}, \sin d \theta )
  \]
  Note that this is well-defined as $d$ is even. 
  As this is a bijection, it is also a homeomorphism, and thus
  $Z \approx \underbrace{S^n \sqcup \cdots \sqcup S^n}_{\frac{d}{2}}$. 
\end{proof}

In particular, we determined the real locus of real forms for the curve
cases:

\begin{thm}
  \begin{enumerate}
    \item Suppose $d$ is odd. Then, the real loci of \[
        ( X_0^d + X_1^d + X_2^d = 0 ),\ ( X_0^d + ( X_1 + \sqrt{- 1}
        X_2 )^d + ( X_1 - \sqrt{- 1} X_2 )^d = 0 )
      \]
      are both homeomorphic to $\P^n ( \R )$. 
    \item Suppose $d$ is even. Then,
      \begin{itemize}
        \item the real locus of $\displaystyle ( X_0^d + X_1^d + X_2^d
          = 0 )$ is $\emptyset$. 
        \item the real locus of $\displaystyle ( - X_0^d + X_1^d + X_2^d
          = 0 )$ is homeomorphic to $S^1$. 
        \item the real locus of $\displaystyle ( X_0^d + ( X_1 +
          \sqrt{- 1} X_2 )^d + ( X_1 - \sqrt{- 1} X_2 )^d = 0 )$
          is homeomorphic to $\underbrace{S^1 \sqcup \cdots
          \sqcup S^1}_{\frac{d}{2}}$. 
      \end{itemize}
  \end{enumerate}
  \label{rl:curve}
\end{thm}

\begin{remark}
  By considering real loci, we see that $F^2_4$ has at least $4$
  non-isomorphic real forms
  \begin{itemize}
    \item $\displaystyle ( X_0^4 + X_1^4 + X_2^4 + X_3^4 = 0 )$,
    \item $\displaystyle ( - X_0^4 + X_1^4 + X_2^4 + X_3^4 = 0 )$,
    \item $\displaystyle ( - X_0^4 - X_1^4 + X_2^4 + X_3^4 = 0 )$, and
    \item $\displaystyle ( X_0^4 + X_1^4 + ( X_2 + \sqrt{- 1} X_3 )^4 +
      ( X_2 - \sqrt{- 1} X_3 )^4 = 0 )$.
  \end{itemize}
  Note that $F^2_4$ is a K3 surface and thus has only finitely many
  real forms up to isomorphisms by \cite{CF20}. However, unlike the other
  cases $d \ge 3$, the automorphism group of $F^2_4$ is still unknown,
  and it is known to be a very difficult problem to determine the
  automorphism group of $F^2_4$.
  For this reason, it is also a very challenging problem to determine all
  real forms of $F^2_4$ or to see at least if there is a real form of
  $F^2_4$ other than these four or not.
  \label{f2_4}
\end{remark}

\medskip\noindent \textbf{Acknowledgements.} I thank Professor J\'{a}nos
Koll\'{a}r for pointing out some crucial mistakes and also, giving some
advice. 

\bibliographystyle{alpha}

\end{document}